\documentclass{ifacconf}

\usepackage{graphicx}      
\usepackage{natbib}        

\usepackage{amsmath, amssymb} 
\usepackage{dsfont}

\usepackage{subcaption}

\usepackage{xcolor}

\definecolor{darkgreen}{RGB}{0,140,0} 

\newcommand{\scal}{\mathbf{\cdot}}                 %
\newcommand{\defeq}{\mathrel{\mathop:}=}

\begin{document}
\begin{frontmatter}

\title{Generalized Lotka-Volterra Model with Species Turnover in a Variable-Basis State Space\thanksref{footnoteinfo}} 

\thanks[footnoteinfo]{This project is supported by the French State as part of Investments for the Future Programme, now integrated into France 2030, and operated by ADEME.}

\author[First]{Arthur Doliveira} 
\author[First]{Christophe Roman} 
\author[First,Second]{Guillaume Graton}
\author[First]{Mustapha Ouladsine}

\address[First]{Aix Marseille Université, Université de Toulon, CNRS, LIS (UMR
7020), Avenue Escadrille Normandie-Niemen, F-13397 Marseille Cedex 20, France (e-mail: arthur.doliveira@lis-lab.fr).}
\address[Second]{Centrale Méditerranée, Technopôle de Château-Gombert, 38 rue
Frédéric Joliot-Curie, F-13451 Marseille Cedex 13, France}

\begin{abstract}                
The state space is a fundamental concept for describing the trajectory of a dynamic system. Depending on its form, it can highlight certain changes over time while ignoring others. This is particularly the case for the spaces associated with theoretical ecology models, notably the generalized Lotka-Volterra (gLV) model, which allows the modeling of interacting populations. The fixed-dimension state space classically used in gLV models does not account for the effective renewal of species through addition, removal, or mutation. To address this limitation, we propose a new variable-base state space, introduced in a previous study. This framework leads to a reformulation of the gLV model within the context of hybrid dynamical systems. To illustrate the approach, we apply the proposed model to the gut microbiota, particularly in the context of bacteriotherapy following antibiotic treatment.
\end{abstract}


\end{frontmatter}

\section{\label{Intro}Introduction}
Theoretical ecology is a fascinating framework that groups concepts and propositions used to study the dynamics of organisms in relation to their abundance and distribution across time and space, as well as their causes (\cite{scheiner2008general,rockwood2015introduction}). These organisms are referred to as ecological populations, and their study focuses on fundamental properties: growth, survival, and reproduction (\cite{rockwood2015introduction}). To analyze these properties, ecologists adopting a mathematical perspective have developed predictive models, including the well-known Lotka–Volterra model. It allows for modeling two foundational principles of modern theoretical ecology: the predator–prey interaction cycle and competition (\cite{wangersky1978lotka}). The former describes an interaction in which one species feeds on another, whereas the latter corresponds to a situation in which two species compete for a shared resource.

The Lotka–Volterra model originates from single-species growth models (\cite{wangersky1978lotka, berryman1992orgins}). In its simplest form, it involves interaction between two species. It was later extended to include additional species, giving rise to the generalized Lotka–Volterra (gLV) model  (\cite{wangersky1978lotka, taylor1988consistent}). This generalized version can accommodate different types of interactions, including those already mentioned (predator–prey and competition) as well as facilitation, in which one species promotes the growth of another (\cite{fort2018predicting}). Based on the principle that populations affect each other’s abundance regardless of the interaction mechanism, \cite{fort2018predicting} used the gLV model to assess its ability to predict species abundance in interacting groups across various taxa, such as plants, algae, and crustaceans. \cite{vaidyanathan2015adaptive} applied adaptive control using this model to a chaotic system comprising two predator populations and one prey population, in order to drive it toward a target state. The model is also applicable to systems involving non-living entities: in technology, to study interactions between a technological product and its components (\cite{zhang2018technology}); and in economics, to model, for example, the temporal evolution of wealth distribution in a society (\cite{malcai2002theoretical}).

Although it is a fascinating model, it exhibits—like many other models in theoretical ecology—a significant limitation that we aim to address in this paper. This limitation is highlighted by \cite{hirsch1984dynamical}. It can be reformulated as a research question in the following way:
\begin{center}
\emph{The state of an ecological system is often defined as the list of populations of a fixed set of species. In reality, some species may disappear while others may appear over time through various processes, such as immigration or mutation. How, then, can one incorporate the phenomena of species emergence and disappearance into the modeling process?}
\end{center}

In the context of the gLV model, two main approaches can be distinguished. The first, classical one represents the state of the system by a fixed-size vector $\mathbf{x} \in \mathbb{R}^n$ ($n \geq 2$), where each coordinate $x_p$ corresponds to the abundance of species $p$ (\cite{jones2020navigation}). However, this approach presents two major limitations. First, the fixed dimension of the vector requires the set of species to be defined \textit{a priori}, which prevents the modeling of certain evolutionary phenomena, such as mutation and immigration. Mutation corresponds to the replacement of a species $p$ by a new species $q$ not initially included, while immigration (or introduction) refers to the arrival of a new species—one not present among the predefined species—into the existing interaction system. Second, the zero value $x_p = 0$, often interpreted as the absence of a species, poses a challenge when modeling microbial communities such as the gut microbiota, where zeros are frequent. According to  \cite{li2021modeling}, these zeros may reflect either the true extinction of a species or the detection limits of sequencing techniques. Failing to distinguish between these possibilities can bias comparisons between models and experimental data. These two limitations are intrinsically linked to the nature of the state space (\cite{hirsch1984dynamical}).

The second approach is proposed by  \cite{taylor1988construction}. In it, a stable system is constructed through the addition and removal of species, which modifies the dimension of the state vector and results in a change of coordinates. These modifications follow algorithmic rules—for example, removing a species when its abundance falls below a threshold—but do not rely on any formal algebraic operations, which is consistent with prior observations of \cite{hirsch1984dynamical} that such operations do not (yet) exist. The resulting changes can be interpreted as a transition from one model to another. Although relevant, this approach remains difficult to reuse, as it does not guarantee the qualitative properties of the model, notably the existence of a solution. Indeed, Taylor defines neither the state space, the operations for changing its dimension, nor the very notion of a solution.

To overcome these limitations, we propose to use a variable-basis state space, introduced in a previous paper (\cite{doliveira:hal-05047454}). Equipped with algebraic operations, we demonstrate its properties. By interpreting the addition and removal of species as discrete events, we propose reformulating the gLV model within the framework of hybrid dynamical systems theory (\cite{goebel2012hybrid}). We then apply this new model to describe the dynamics of the gut microbiota, with particular attention to bacteriotherapy following antibiotic treatment.

The gut microbiota is a complex and dynamic community of diverse and abundant microorganisms inhabiting the gastrointestinal tract (\cite{thursby2017introduction, panda2014structure}). These microorganisms interact with each other, contributing to the stability and growth of the community (\cite{lin2022linear, luo2024progress}). Perturbations in its composition—for example due to infections or antibiotic administration—or changes in microbial interactions can affect health, contributing to diseases such as obesity, inflammatory bowel disease, or HIV (\cite{lin2022linear}). Many studies modeling the dynamics of the microbiota under these perturbations use the gLV model (\cite{jones2020navigation, li2021modeling}). To restore microbial balance, there is a clinical approach called bacteriotherapy, which involves directly adding or removing specific microbial species to achieve a target microbial state (\cite{jones2020navigation}). This application clearly illustrates that the number of species can evolve over time.

The remainder of the paper is structured as follows. In Section \ref{VarBas}, we describe the variable-basis space together with its associated algebraic structure. In Section \ref{HybMod}, we formulate the gLV model in this variable-basis space within the framework of hybrid dynamical systems theory. In Section \ref{Bact}, we apply this model to a scenario of bacteriotherapy following antibiotic treatment. Finally, in Section \ref{Conc}, we present the conclusion.

\section{\label{VarBas}Variable-Basis State-Space}
The variable-basis space considered here, inspired by the variable-dimension space of \cite{cheng2018linear}, is one of the three components of the space introduced via the Cartesian product in (\cite{doliveira:hal-05047454}). It has two internal composition laws and one external law. In (\cite{doliveira:hal-05047454}), it is stated without proof that, with either internal law and the external law, this space forms a complete and simple semi-vector space over the real numbers $\mathbb{R}$. We present a proof here, starting with a description of the space and its composition laws.

Let us assume that $\mathcal{I} \subset \mathbb{N}$ is a countable set of labels associated, in the context of the intestinal microbiota, with microbial species preserved in the global intestinal microbiota repository (\cite{dominguez2025microbiota}). Let $\mathcal{B} = \{\mathbf{b}_i \mid i \in \mathcal{I}\}$ be a set of linearly independent vectors indexed by $\mathcal{I}$. We denote by $2^{\mathcal{B}}$ the power set of $\mathcal{B}$, so that for any $B \in 2^{\mathcal{B}}$, we have $B \subset \mathcal{B}$. Both the empty set $\varnothing$ and $\mathcal{B}$ itself belong to $2^{\mathcal{B}}$. For any $B \in 2^{\mathcal{B}}$, we define $I(B) = \{i \in \mathcal{I} \mid \mathbf{b}_i \in B\}$ as the set of labels associated with the vectors in $B$. Finally, $B$ generates a subspace denoted $\mathcal{V}(B)$, defined as
\begin{equation}
\label{attVector}
\mathcal{V}(B) \!= \!\operatorname{span}(B) \!=\! \left\{ \sum_{i \in I(B)} \!\!\!\alpha_i \mathbf{b}_i \mid \alpha_i \in \mathbb{R}, \mathbf{b}_i \in B \right\}.
\end{equation}

\vspace{-2 mm}

For any $B \subset \mathcal{B}$, the space $\mathcal{V}(B)$, equipped with vector addition, is a vector space over $\mathbb{R}$. When $B = \varnothing$, $\mathcal{V}(B)$ is the trivial subspace, whose only element is the null vector $\mathbf{0}$, and it has dimension zero. When $B = \mathcal{B}$, we speak of the maximal-basis space. We denote by $\mathbf{x}_{\mathcal{B}}^{0}$ the maximal zero-coordinate vector, whose components satisfy $x_k = 0$ for all $k \in I(\mathcal{B})$.

\begin{defn}
The variable-basis space generated by the subsets of $\mathcal{B}$ is the space denoted by $\mathcal{V}$ and given by
\[
\mathcal{V} \defeq \bigcup\limits_{B \in 2^{\mathcal{B}_v}} \mathcal{V}(B).
\]
An element of $\mathcal{V}$, denoted $\mathbf{x}_B$, is said to be defined on the basis $B \in 2^{\mathcal{B}}$, and can be written as
\[
\mathbf{x}_B \defeq \sum\limits_{k \in I(B)} x_k \mathbf{b}_k, \quad \text{where } x_k \in \mathbb{R}.
\]
\end{defn}

\begin{defn}
\label{CompInterLaw}
Let $\cup$ and $\cap$ denote the union and intersection operations from set theory. We denote by $+_\star$, with  $\star \in \{\cup, \cap\}$, an operation arising from the mapping  $+_\star : \mathcal{V} \times \mathcal{V} \to \mathcal{V}$. For any pair  \((\mathbf{x}_{B}, \mathbf{y}_{D}) \in \mathcal{V}^2\), where   \(\mathbf{x}_{B} = \sum_{k \in I(B)} x_k \mathbf{b}_k\) and \(\mathbf{y}_{D} = \sum_{l \in I(D)} y_l \mathbf{b}_l\), the operation $+_\star$ is defined by
\[
\mathbf{x}_{B} +_{\star} \mathbf{y}_{D}
= \mathbf{x}_{B \star D} + \mathbf{y}_{D \star B},
\quad \text{avec } \star \in \{\cup, \cap\},
\]
where \(\mathbf{x}_{B \star D}\) and \(\mathbf{y}_{D \star B}\) denote the projections of \(\mathbf{x}_{B}\) and \(\mathbf{y}_{D}\)  onto the basis \(B \star D = D \star B\). More explicitly,
\[
\mathbf{x}_{B} +_{\star} \mathbf{y}_{D}
= \sum_{p \in I(B \star D)} (x_p + y_p) \mathbf{b}_p,
\quad \text{with } \star \in \{\cup, \cap\}.
\]
If $\star = \cup$, then for any \(p \in I(B)\) such that \(p \notin I(D)\) (and vice versa), let \(y_p = 0\) (ou \(x_p = 0\), respectively).
\end{defn}

The operation $+_{\cup}$ is called the sum over the united basis, and $+_{\cap}$ is called the sum over the common basis. Using the first, the result of the operation has an expanded basis (or dimension), whereas with the second, the basis (or dimension) is reduced.
\begin{defn}
\label{LoiMulti}
Let $\mathbb{R}$ be the field of real numbers. The external multiplicative law is the mapping
\[
\begin{aligned}
\scal : \mathbb{R} \times \mathcal{V} &\longrightarrow \mathcal{V} \\
(\alpha, \mathbf{x}_{B}) &\longmapsto \alpha \scal \mathbf{x}_{B} = \sum_{k \in I(B)} \alpha x_k \mathbf{b}_k,
\end{aligned}
\]
where, for all $k \in I(B)$, $x_k \in \mathbb{R}$. For all $\alpha,\beta \in \mathbb{R}$, $\mathbf{x}_{B}, \mathbf{y}_{D} \in \mathcal{V}$ and for each $\star \in \{\cup, \cap\}$ the law $\scal$ satisfies:
\begin{enumerate}
\item Distributivity with respect to the internal law $+_\star$:
\[
\alpha \scal (\mathbf{x}_{B} +_\star \mathbf{y}_{D}) = \alpha \scal \mathbf{x}_{B} +_\star \alpha \scal \mathbf{y}_{D};
\]
\item Distributivity with respect to addition in $\mathbb{R}$:
\[
(\alpha + \beta) \scal \mathbf{x}_{B} = \alpha \scal \mathbf{x}_{B} +_\star \beta \scal \mathbf{x}_{B};
\]
\item Associativity with respect to real multiplication:
\[
(\alpha \beta) \scal \mathbf{x}_{B} = \alpha \scal (\beta \scal \mathbf{x}_{B})
\]
\item Existence of a neutral element:
\[
1 \scal \mathbf{x}_{B} = \mathbf{x}_{B}.
\]
\end{enumerate}
\end{defn}

\begin{thm}
\label{SemiSpaceVeCup}
The triple $(\mathcal{V}, +_\cup, \scal)$ is a complete and simple semi-vector space over $\mathbb{R}$, whose neutral element is the zero vector $\mathbf{0}$ of null dimension.
\end{thm}

\begin{thm}
\label{SemiSpaceVeCap}
The triple $(\mathcal{V}, +_\cap, \scal)$ is a complete and simple semi-vector space over $\mathbb{R}$, whose neutral element is the maximal zero-coordinate vector $\mathbf{x}_{\mathcal{B}}^{0}$.
\end{thm}

\begin{note}
\label{SemiSurCorps}
A semi-vector space is said to be complete if it admits a neutral element, and simple if it contains elements without inverses (\cite{janyvska2007semi}). It is generally defined over the semiring $\mathbb{R}^+$ rather than over the field $\mathbb{R}$, but Note 1 in (\cite{doliveira:hal-05047454}) shows that extending it to $\mathbb{R}$ is possible for certain spaces, such as the space of graphs. The same holds for $\mathcal{V}$, equipped with $+_\star$ for any $\star \in \{\cup, \cap\}$, together with the operation $\scal$.
\end{note}

\subsection{Proof of Theorem \ref{SemiSpaceVeCup}}
In order for the triple $(\mathcal{V}, +_\cup, \scal)$ to be considered a complete and simple semi-vector space over the field of real numbers $\mathbb{R}$ (see \textit{Note} \ref{SemiSurCorps}), it is necessary that it satisfies the properties stated in Definition 10 of \cite{doliveira:hal-05047454} (or Definition 1.1 of \cite{janyvska2007semi}).  Part of these properties concerns the multiplicative law, which, by definition, is satisfied in \textit{Definition} \ref{LoiMulti}. The other part consists in showing that $(\mathcal{V}, +_\cup)$ is a commutative monoid. Commutativity is straightforward to verify; we will therefore focus on the existence of a neutral element and on the fact that $\mathcal{V}$ contains elements without inverses.
\begin{prop}
The pair $(\mathcal{V}, +_\cup)$ has the zero vector $\mathbf{0}$ of null dimension as its neutral element.
\end{prop}
\begin{pf}
An arbitrary element $\mathbf{e}_D \in \mathcal{V}$ is said to be neutral if and only if, for every $\mathbf{x}_B \in \mathcal{V}$, we have
\begin{equation}
\label{neutrecup}
\mathbf{x}_B +_\cup \mathbf{e}_D = \mathbf{x}_B,
\end{equation}
knowing that $+_\cup$ is commutative. Equation \eqref{neutrecup} implies that $B \cup D = B$ (see \textit{Definition} \ref{CompInterLaw}). Thus, either $D = \varnothing$, in which case $\mathbf{e}_D = \mathbf{0}$ (the zero vector), or $D = B \neq \varnothing$, in which case the components of the vector $\mathbf{e}_D$ need to be determined. Suppose $D = B$. Then we have
\[
\mathbf{x}_B = \sum_{k \in I(B)} x_k \mathbf{b}_k \quad \text{and} \quad
\mathbf{e}_B = \sum_{k \in I(B)} e_k \mathbf{b}_k,
\]
with $x_k, e_k \in \mathbb{R}$ for all $k \in I(B)$. From \eqref{neutrecup}, we obtain
\[
\sum_{k \in I(B)} (x_k + e_k) \mathbf{b}_k = \sum_{k \in I(B)} x_k \mathbf{b}_k.
\]
Thus, for all $k \in I(B)$, we have $x_k + e_k = x_k$, so $e_k = 0$. The zero vector $\mathbf{0}$ and the vector $\mathbf{e}_B^0$, for which all $e_k = 0$ for $k \in I(B)$, are the only ones satisfying condition \eqref{neutrecup}.  However, only $\mathbf{0}$ is a neutral element of $\mathcal{V}$ with respect to $+_\cup$. Indeed, let $\mathbf{y}_C \in \mathcal{V}$ with $C \neq B$ and $C \neq \varnothing$; then $\mathbf{y}_C +_\cup \mathbf{e}_B^0 \neq \mathbf{y}_C$. Hence, the only element satisfying \eqref{neutrecup} is the zero vector $\mathbf{0}$ of null dimension.
\end{pf}
\begin{prop}
The pair $(\mathcal{V}, +_\cup)$ contains elements without inverses.
\end{prop}
\begin{pf}
An element $\mathbf{u}_D \in \mathcal{V}$ is said to be the inverse of an element $\mathbf{x}_B \in \mathcal{V}$ if and only if $\mathbf{x}_B +_\cup \mathbf{u}_D = \mathbf{0}$, knowing that $+_\cup$ is commutative. This condition implies that $B \cup D = \varnothing$, which in turn leads to $B = \varnothing$ and $D = \varnothing$. Therefore, the only element admitting an inverse is the zero vector $\mathbf{0}$.
\end{pf}

The set of arguments developed in this subsection, together with \textit{Note} \ref{SemiSurCorps}, leads to the conclusion that $(\mathcal{V}, +_\cup, \scal)$ forms a complete and simple semi-vector space over $\mathbb{R}$, whose neutral element is the zero vector $\mathbf{0}$ of null dimension.

\subsection{Proof of Theorem \ref{SemiSpaceVeCap}}
As in the proof of \textit{Theorem} \ref{SemiSpaceVeCup}, we will focus here on the existence of a neutral element and on the fact that $\mathcal{V}$ contains elements without inverses.
\begin{prop}
The pair $(\mathcal{V}, +_\cap)$ has the maximal zero-coordinate vector $\mathbf{x}_{\mathcal{B}}^0$ as its neutral element.
\end{prop}
\begin{pf}
An arbitrary element $\mathbf{e}_D \in \mathcal{V}$ is said to be neutral if and only if, for every $\mathbf{x}_B \in \mathcal{V}$, we have
\begin{equation}
\label{neutrecap}
\mathbf{x}_B +_\cap \mathbf{e}_D = \mathbf{x}_B,
\end{equation}
knowing that $+_\cap$ is commutative. Equation \eqref{neutrecap} leads to $B \cap D = B$ (see \textit{Definition } \ref{CompInterLaw}), which implies that $B \subseteq D$. In order for this relation to hold for every $B \in 2^{\mathcal{B}}$, it is necessary that $D = \mathcal{B}$. Consider the extreme case $B = D = \mathcal{B}$ rather than $B \subset D$. Under these conditions, we have
\vspace{- 2mm}
\[
\mathbf{x}_{\mathcal{B}} = \sum_{k \in I(\mathcal{B})} x_k \mathbf{b}_k
\quad \text{and} \quad
\mathbf{e}_{\mathcal{B}} = \sum_{k \in I(\mathcal{B})} e_k \mathbf{b}_k.
\]
From \eqref{neutrecap}, it follows that $x_k + e_k = x_k$, so $e_k = 0$ for all $k \in I(\mathcal{B})$. Thus, the maximal zero-coordinate vector $\mathbf{x}_{\mathcal{B}}^0$ is the neutral element of $\mathcal{V}$ with respect to the operation $+_\cap$.
\end{pf}

\begin{prop}
The pair $(\mathcal{V}, +_\cap)$ contains elements without inverses.
\end{prop}
\begin{pf}
An element $\mathbf{u}_D \in \mathcal{V}$ is said to be the inverse of an element $\mathbf{x}_B \in \mathcal{V}$ if and only if
\begin{equation}
\label{InvCap}
\mathbf{x}_B +_\cap \mathbf{u}_D = \mathbf{x}_{\mathcal{B}}^0,
\end{equation}
knowing that $+_\cap$ is commutative. Equation \eqref{InvCap} leads to $B \cap D = \mathcal{B}$, which necessarily implies that $B = D = \mathcal{B}$. Thus, for
\[
\mathbf{x}_{\mathcal{B}} = \sum_{k \in I(\mathcal{B})} x_k \mathbf{b}_k
\quad \text{and} \quad
\mathbf{u}_{\mathcal{B}} = \sum_{k \in I(\mathcal{B})} u_k \mathbf{b}_k,
\]
from \eqref{InvCap} we have $x_k + u_k = 0$, hence $u_k = -x_k$ for all $k \in I(\mathcal{B})$. Therefore, only the elements of $\mathcal{V}(\mathcal{B})$ admit an inverse, and not the other elements of $\mathcal{V}$.
\end{pf}

Just as in the proof of \textit{Theorem} \ref{SemiSpaceVeCup}, the set of arguments developed in this subsection, together with \textit{Note} \ref{SemiSurCorps}, leads to the conclusion that $(\mathcal{V}, +_\cap, \scal)$ forms a complete and simple semi-vector space over $\mathbb{R}$, whose neutral element is the maximal zero-coordinate vector $\mathbf{x}_{\mathcal{B}}^0$.

\section{\label{HybMod}Hybrid-Dynamics gLV Model}
The evolution of a system in $\mathcal{V}$, whose state is a vector that may change basis or dimension over time, can be addressed within the framework of hybrid dynamical systems theory. The change of basis is approached as a discrete event, whereas the coordinates evolve continuously. Here we adopt the formalism of  \cite{goebel2012hybrid} and introduce below the elements required for the model.
\begin{defn}
A subset $\mathcal{T} \subset \mathbb{R}_{\ge 0} \times \mathbb{N}$ is a compact hybrid time domain if
\vspace{-2 mm}
\[
\mathcal{T} \defeq \bigcup_{k = 0}^{K-1} \big([\tau_k, \tau_{k+1}], k\big)
\]
for a finite sequence of times $0=\tau_0 \le \tau_1 \le \dots \le \tau_K$.  The subset $\mathcal{T} \subset \mathbb{R}_{\ge 0} \times \mathbb{N}$ is a hybrid time domain if it is the union of a nondecreasing sequence of compact hybrid time domains, that is, if $\mathcal{T}$ is the union of compact hybrid time domains $\mathcal{T}_k$ satisfying:
\[
\mathcal{T}_0 \subset \mathcal{T}_1 \subset \dots \subset \mathcal{T}_k \subset \dots
\]
\end{defn}

\begin{defn}
The state of a hybrid system in $\mathcal{V}$ at time $(t,k) \in \mathcal{T}$ is an element $\mathbf{x}_{B(k)}(t,k) \in \mathcal{V}(B_k) \subset \mathcal{V}$, where $B(k) \in 2^{\mathcal{B}}$ denotes the basis of the state vector after the $k$-th jump. For simplicity, we write $\mathbf{x}_{B_k}(t) \defeq \mathbf{x}_{B(k)}(t,k)$. This state satisfies that, for every $k \in \mathbb{N}$, the function $t \mapsto \mathbf{x}_{B_k}(t))$ is absolutely continuous on the interval $[\tau_k, \tau_{k+1}]$, for every $([\tau_k, \tau_{k+1}], k) \subset \mathcal{T}$.
\end{defn}

\begin{defn}
The disturbance is the pair $\boldsymbol{\omega} = (u, \mathbf{v}_{L_k}) \in \mathbb{R} \times \mathcal{V})$, where $L_k \in 2^{\mathcal{B}}$. It is such that, for every $k \in \mathbb{N}$, the function $t \mapsto \boldsymbol{\omega}(t) = (u(t), \mathbf{v}_{L_k}(t))$ is Lebesgue measurable and locally essentially bounded on the interval $[\tau_k, \tau_{k+1}]$, for every $([\tau_k, \tau_{k+1}], k) \subset \mathcal{T}$.
\end{defn}

The hybrid-dynamics evolution model is based on hybrid equations, that is, on a set composed of differential equations and difference equations, subject to constraints, of the following form:
\begin{equation}
\label{HybridEq}
\begin{cases}
 (\mathbf{x}_{B_k}(t),u(t))\!\in\! F, & \!\! \dot{\mathbf{x}}_{B_k}(t)\! =\!f(\mathbf{x}_{B_k}(t),u(t)),  \\
 (\mathbf{x}_{B_k}(t),\mathbf{v}_{L_k}(t))\!\in\! J, & \!\! \mathbf{x}^+_{B_{k+1}}(t) = j(\mathbf{x}_{B_k}(t),\mathbf{v}_{L_k}(t)),
\end{cases}
\end{equation}
where $F \subset \mathcal{V} \times \mathbb{R}$ and $J \subset \mathcal{V}^2$ are the flow set and the jump set, respectively, while $f$ and $j$ are the flow map and the jump map, respectively. Note that, for any $B \in 2^{\mathcal{B}}$, the flow map is well defined as a mapping $f : \mathcal{V}(B) \times \mathbb{R} \to \mathcal{V}(B)$.  It describes the continuous dynamics of the system, in this case that of the gLV model. Its vector form, derived from the description proposed in \cite{jones2020navigation}, for a microbiota dynamic under the effect $u(t)$ of an antibiotic, is written as follows:
\begin{equation}
\label{gLVec}
\begin{split}
f(\mathbf{x}_{B_k}(t),u(t)) = \mathbf{x}_{B_k}(t) \odot\!  \left[ \rho_{B_k} +\right.&  W_{B_k} \mathbf{x}_{B_k}(t) \\ 
&\left.+ u(t)\scal \boldsymbol{\varepsilon}_{B_k}\right]
\end{split}
\end{equation}

\vspace{-2 mm}

where $\odot$ denotes the Hadamard product, also called element-wise multiplication. The vector $\boldsymbol{\rho}_{B_k}$ contains the growth rates of the species $i \in I(B_k)$. The matrix $W_{B_k}$ is an interaction matrix whose entries $(p,q) \in I(B_k) \times I(B_k)$ encode the ecological effects of species $q$ on species $p$. Finally, the vector $\boldsymbol{\varepsilon}_{B_k}$ collects the susceptibilities $\varepsilon_p$, for $p \in I(B_k)$, of species $p$ to the antibiotic. The jump map, on the other hand, is a mapping $j: \mathcal{V}^2 \to \mathcal{V}$, with the notation $\mathbf{x}^+_{B_{k+1}}(t)$ denoting $\lim_{t \to \tau_{k+1}^+} \mathbf{x}_{B_{k+1}}(t)$. Let the jump set be $J = J_{\nearrow} \cup J_{\searrow} \cup (J_{\star} \times \mathcal{V})$, with definitions given in equation \eqref{jumpset}. We then consider three implementations of this function:
\begin{equation}
\begin{split}
j(&\mathbf{x}_{B_k}(t),\mathbf{v}_{L_k}(t)) =\\
&\begin{cases}
\mathbf{x}_{B_k}(t) +_\cup \mathbf{v}_{L_k}(t), & \text{if } (\mathbf{x}_{B_k}(t),\mathbf{v}_{L_k}(t)) \in J_{\nearrow} ,\\
\mathbf{x}_{B_k}(t) +_\cap \mathbf{v}_{L_k}(t), & \text{if } (\mathbf{x}_{B_k}(t),\mathbf{v}_{L_k}(t)) \in J_{\searrow} ,\\
\tilde{j}(\mathbf{x}_{B_k}(t)), & \text{if }  \mathbf{x}_{B_k}(t) \in J_\star.
\end{cases}
\end{split}
\end{equation}
The vector $\mathbf{v}_{L_k}$ models an external perturbation of the system. In the first case, it increases the size of the state vector by adding species; in the second, it decreases it by removing species. The third situation corresponds to a constraint linked to the intrinsic dynamics of the system, described by the function $f$ in \eqref{gLVec}. It illustrates, for example, the disappearance or appearance of species, or the mutation of an existing species under the effect of antibiotics. These changes occur when the abundances cross certain thresholds, as described in \cite{taylor1988construction}. To illustrate, let us assume $J_\star = J_{+} \cup J_{-}$, whose definitions are given in \eqref{jumpset}, with 
\begin{equation}
\label{autojump}
\begin{split}
 \tilde{j}(\mathbf{x}_{B_k}(t))& =\\
 &\begin{cases}
\mathbf{x}_{B_k}(t) +_\cup g(\mathbf{x}_{B_k}(t)), & \text{if } \mathbf{x}_{B_k}(t) \in J_{+} ,\\
\mathbf{x}_{B_k}(t) +_\cap h(\mathbf{x}_{B_k}(t)) & \text{if } \mathbf{x}_{B_k}(t)\in J_{-},
\end{cases}
\end{split}
\end{equation}
where the function $g(\mathbf{x}_{B_k}(t)) = \xi^+ \scal \mathbf{1}(\mathbf{x}_{B_k}(t)) +_\cup \mathbf{y}_D$, with $\mathbf{y}_D \in \mathcal{V}$ such that $B_k \cap D = \varnothing$ and $\xi^+ \in \mathbb{R}$, represents a small perturbation induced by the appearance of species. In the stochastic framework, one can assume that the set $D \in 2^{\mathcal{B}}$ is associated with a conditional probability given $B_k$ and the values of the coordinates of $\mathbf{x}_{B_k}$. The function $h(\mathbf{x}_{B_k}(t)) = \xi^- \scal \mathbf{1}_H$, with $H \subset B_k$ and $\xi^- \in \mathbb{R}$, represents the small perturbation induced by the disappearance of species. The function $\mathbf{1}(\mathbf{x}_{B_k})$ constructs the unit vector on the basis $B_k$ from $\mathbf{x}_{B_k}$, and $\mathbf{1}_H$ is a unit vector on the basis $H$. The sets used to construct the jump set $J$ can be defined as
\begin{equation}
\label{jumpset}
\begin{aligned}
J_{\nearrow} &= \left\{(\mathbf{x}_{B},\mathbf{y}_{D}) \in \mathcal{V}^2 \mid D\neq \varnothing, \exists H \subseteq D \text{ s.t. } H \not\subset B\right\}, \\
J_{\searrow} &= \left\{(\mathbf{x}_{B},\mathbf{y}_{D}) \in \mathcal{V}^2 \mid D\neq \varnothing \text{ s.t. } D \subset B\right\}, \\
J_{+} &= \left\{\mathbf{x}_{B}\in \mathcal{V} \mid \exists J \subseteq I(B)  \text{ s.t. }  x_j \geq \alpha > 0~~ \forall j \in J\right\}, \\
J_{-} &= \left\{\mathbf{x}_{B}\in \mathcal{V} \mid \exists J \subseteq I(B)  \text{ s.t. }  0 < x_j \leq \beta ~~ \forall j \in J\right\}.
\end{aligned}
\end{equation}

\begin{note}
Verifying the constraint $\mathbf{x}_{B_k}(t) \in J_+$ must be done carefully, as it depends on the existence of at least one coordinate of $\mathbf{x}_{B_k}(t)$ that exceeds the threshold $\alpha > 0$. According to \eqref{autojump}, such a crossing triggers the addition of new species and leads to the updated state $\mathbf{x}_{B_{k+1}}(\cdot)$, of higher dimension. As long as the coordinate responsible for the jump remains above $\alpha$, the state stays within $J_+$, potentially generating an infinite number of additions—this is a Zeno effect (\cite{goebel2008zeno}) on the state dimension. To exit $J_+$, this coordinate must either  switch to a different dynamics after the first jump or be considered only once as responsible for the jump throughout the simulation. 
\end{note}
The concept of solution and the existence result considered are based on ongoing work, which consists of a major revision of the article \cite{doliveira:hal-05047454}.

\section{\label{Bact}Application to Bacteriotherapy}
The numerical implementation of this model for the dynamics of the microbiota, in particular for bacteriotherapy following antibiotic treatment as described in Section \ref{Intro}, consists here of specifying all the data involved in the model. In this implementation, the scientific names of the species are not considered; instead, their labels are indexed by the set $\mathcal{I} \subset \mathbb{N}$, referred to as the universe. Our universe $\mathcal{I}$ is partitioned into two subsets, $\mathcal{I} = \mathcal{I}_\text{n} \cup \mathcal{I}_\text{u}$, with $\mathcal{I}_\text{n} \cap \mathcal{I}_\text{u} = \varnothing$, where $\mathcal{I}_\text{n}$ contains the known and catalogued species, while $\mathcal{I}_\text{u}$ contains species that are neither known nor catalogued.

In this implementation, we consider $\mathcal{I}_\text{n} = \{1,2,\dots,11\}$, corresponding to the labels associated with the $11$ categories of species catalogued by  \cite{stein2013ecological}, in the order of their appearance in (\cite{stein2013ecological}, Figure 2). For example, species 1 is Barnesiella and species 9 is Clostridium difficile. The growth rates, antibiotic susceptibilities, and the ecological interaction matrix of these species are those shown in Figure 2 of \cite{stein2013ecological}. Our implementation is limited to $\mathcal{I}_\text{n}$. Suppose the initial state is given by
\begin{equation}
\label{initialcondition}
\mathbf{x}_{B_0}(0) = 0.7 \mathbf{b}_1 + 0.3 \mathbf{b}_2+ 1.2 \mathbf{b}_4 + 0.5 \mathbf{b}_5,
\end{equation}
the antibiotic effect by $u(t) = 1$ for $0 \le t \le 4$ and $u(t) = 0$ otherwise, and the bacteriotherapy by
\begin{equation}
\label{bacteriotherapy}
  \begin{split}
  &\mathbf{v}_{L_k}(t) =\\
    &\begin{cases}
        2.10^{-4} \mathbf{b}_5 + 0.85\mathbf{b}_9, & \text{if } t = 190,\\
      0.2\mathbf{b}_3 + 1.3 \mathbf{b}_5 + 0.96\mathbf{b}_8 , & \text{if } t = 330,\\
      7.10^{-3} \mathbf{b}_1 + 3.10^{-3} \mathbf{b}_2 + 5.10^{-3}\mathbf{b}_4 , & \text{if } t = 560,\\
      \mathbf{0}, & \text{otherwise}.
    \end{cases}
  \end{split}
\end{equation}

\vspace{- 2 mm}

For an autonomous jump, our implementation will be limited to the extinction of species induced by the set $J_-$ described in \eqref{jumpset}. We then set $\beta = 1 \cdot 10^{-6}$. Four different simulations are performed, and the results are shown in figures \ref{fig:microbiotadyn} and \ref{fig:microbiotaDynBacterio}. The input $\mathbf{v}_{L_k}$ at $t = 560$ represents a species removal (see $J_{\searrow}$ in \eqref{jumpset} and Fig.~\ref{fig:microbiotaDynBacterio}) and indicates which species to retain in the system state. Various numerical methods can be used, including the semi-implicit Crank–Nicolson and an explicit scheme\footnote{Readers may contact the authors to obtain the developed C++ code at the following email address: \texttt{arthur.doliveira@lis-lab.fr}.}.

In general, the results presented agree well with those of \cite{jones2020navigation}, according to which the system transitions from one equilibrium state to another as species are added or removed. The difference in order of magnitude with \cite{jones2020navigation} is mainly due to a scaling factor, as described by  \cite{stein2013ecological}. Fig. \ref{fig:sub1}, compared to Fig. \ref{fig:sub2}, is intended primarily to show the effect of the antibiotic on microbiota evolution, namely the slowing down of its ability to reach an equilibrium state. Fig. \ref{fig:sub3}, compared to Fig. \ref{fig:sub4}, illustrates the necessity of distinguishing a near-zero abundance ($x_i \approx 0$) from the effective absence of a species, as reflected in the evolution of the associated basis vectors. As noted in Section \ref{Intro}, zero or near-zero values are common in biology. In Fig. \ref{fig:sub3}, without removing species that fall below a threshold, species 2 — whose abundance had become very low — increases again after all species except 1, 2, and 4 are eliminated at $t = 560$ (see \eqref{bacteriotherapy}). In contrast, in Fig. \ref{fig:sub4}, where the constraint $J_-$ is activated (see \eqref{jumpset}), species 2 does not reappear, as it is considered extinct around day 500 after crossing the threshold $\beta$. In Fig.~\ref{fig:sub4}, we consider a variant of the input at $t = 560$, given by $\mathbf{v}_{L_k}(t) = 7 \cdot 10^{-3} \mathbf{b}_1 + 5 \cdot 10^{-3} \mathbf{b}_4$, as species 2 is no longer present (see $J_{\searrow}$ in \eqref{jumpset}).

\begin{figure}[t]
    \centering
    
    \begin{subfigure}[b]{0.5\textwidth}
        \centering
        \includegraphics[width=0.88\textwidth, height = 4.5 cm]{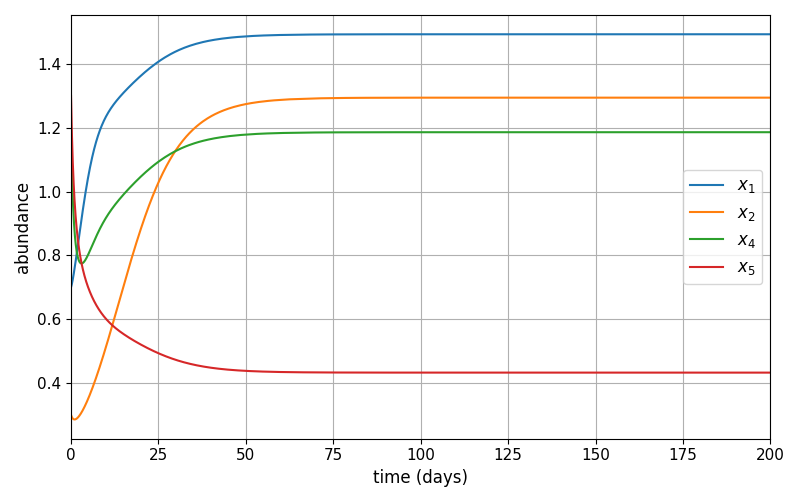}
        \caption{}
        \label{fig:sub1}
    \end{subfigure}
    
    \vspace{-0.1cm} 
    
    \begin{subfigure}[b]{0.5\textwidth}
        \centering
        \includegraphics[width=0.88\textwidth, height = 4.5 cm]{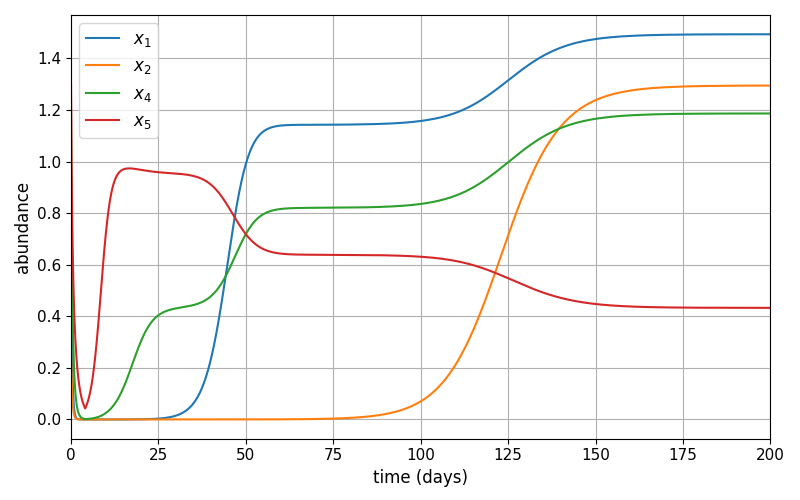}
        \caption{}
        \label{fig:sub2}
    \end{subfigure}
    
    \caption{Evolution of a microbiota simulated with the gLV model \eqref{HybridEq} and initial condition \eqref{initialcondition}: (a) without the antibiotic effect  $u(t)$; (b) with the antibiotic effect.}
    \label{fig:microbiotadyn}
\end{figure}

\begin{figure}[t]
    \centering
    
    \begin{subfigure}[b]{0.5\textwidth}
        \centering
        \includegraphics[width=0.88\textwidth]{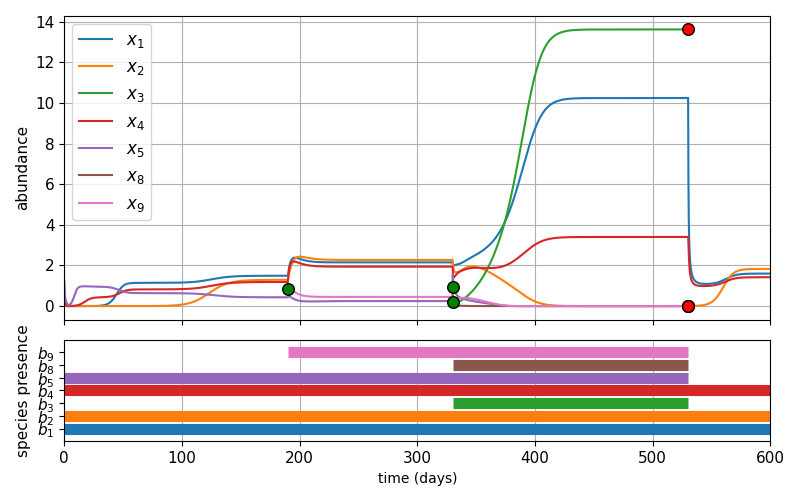}
        \caption{}
        \label{fig:sub3}
    \end{subfigure}
    
    \vspace{-0.1cm} 
    
    \begin{subfigure}[b]{0.5\textwidth}
        \centering
        \includegraphics[width=0.88\textwidth]{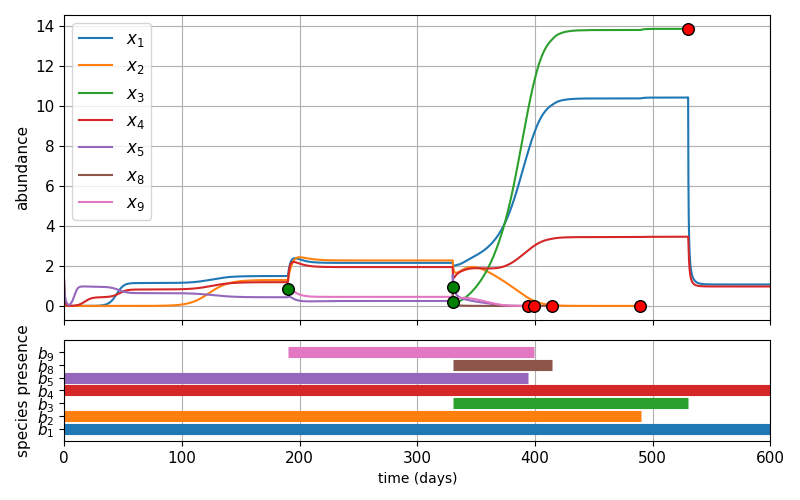}
        \caption{}
        \label{fig:sub4}
    \end{subfigure}
    
    \vspace{-0.3cm} 

    \caption{Evolution of a microbiota simulated with the gLV model \eqref{HybridEq} and initial condition \eqref{initialcondition} over 600 days, under the antibiotic effect $u(t)$ and bacteriotherapy $\mathbf{v}_{L_k}$ : (a) without activation of the jump set $J_-$; (b) with activation. \textcolor{darkgreen}{\textbullet} indicate species appearance, \textcolor{red}{\textbullet} indicate disappearance.}
    \label{fig:microbiotaDynBacterio}
\end{figure}

\section{\label{Conc}Conclusion}
This paper models species turnover in an ecological system governed by the gLV dynamics by introducing a variable-basis state space and its associated algebraic structure. This leads to a reformulation of the gLV model within the framework of hybrid dynamical systems. Application to microbiota dynamics illustrates the turnover mechanism. Future work includes parameter calibration using experimental data and a possible stochastic reformulation.


\bibliography{ifacconf}             
                                                   







\end{document}